\documentclass[12pt]{amsart}

\usepackage{amssymb}

\usepackage{enumerate}

\usepackage{graphicx}

\makeatletter
\@namedef{subjclassname@2010}{%
  \textup{2010} Mathematics Subject Classification}
\makeatother

\usepackage{etex}
\usepackage{txfonts}
\usepackage{float}
\usepackage{tikz}
\usetikzlibrary{matrix}
\input xy
\xyoption{all}

\newtheorem{thm}{Theorem}[section]
\newtheorem{cor}[thm]{Corollary}
\newtheorem{lem}[thm]{Lemma}
\newtheorem{prop}[thm]{Proposition}
\newtheorem{prob}[thm]{Problem}

\theoremstyle{definition}
\newtheorem{defin}[thm]{Definition}
\newtheorem{rem}[thm]{Remark}

\numberwithin{equation}{section}

\def \N {\mathbb N}

\def \Z {\mathbb Z}
\def \R {\mathbb R}

\def \X {\mathcal{X}}
\def \a {\alpha }
\def \b {\beta}

\def \w {\omega}

\begin{document}


\baselineskip=17pt



\title[Uniformly rigid models for rigid actions]{Uniformly rigid models for rigid actions}

\author[S. Donoso]{Sebasti\'an Donoso}
\address{ Departamento de Ingenier\'{\i}a
	Matem\'atica \\  Universidad de Chile\\ Beauchef 851, Santiago, Chile.}
\email{sdonosof@gmail.com}

\author[S. Shao]{Song Shao}
\address{Wu Wen-Tsun Key Laboratory of Mathematics USTC, Chinese Academy of Sciences and
	Department of Mathematics\\ University of Science and Technology of China \\
	Hefei, Anhui, 230026, P.R. China.}
\email{songshao@ustc.edu.cn}

\date{}

\begin{abstract} In this article we show that any ergodic non-periodic rigid system can be topologically realized by a uniformly rigid and (topologically) weak mixing topological dynamical system.
\end{abstract}



\maketitle

\section{Introduction}

A fundamental problem in ergodic theory and topological dynamics is the one of recurrence. In this paper we are interested in the relation in the measurable and topological context of a special strong form of recurrence, the so called rigidity property. The main result states that any ergodic rigid system can be topologically realized in a uniformly rigid and topologically weakly mixing system.

A measure preserving system $(X,\mathcal{X},\mu,T)$ is rigid if there exists an increasing sequence $(n_i)_{i\in \N}$ in $\N$ such that $T^{n_i}$ converges to the identity in the strong operator topology. This means that for any $f\in L^{2}(\mu)$ one has that $\|f-f\circ T^{n_i}\|_{2}$ goes to 0 as $i$ goes to infinity. This is also equivalent to say that $\mu(A\cap T^{n_i}A)$ converges to $\mu(A)$ for any measurable set $A$. Usually one refers to the sequence $(n_i)_{i\in \N}$ as a rigidity sequence of $(X,\mathcal{X},\mu,T)$. Very recently, nice results about what kind of sequences can be rigidity sequences for weakly mixing systems have been given \cite{BJLR,FK,FT}.

Topological analogues of rigidity were introduced by Glasner and Maon \cite{GM}. A topological dynamical system $(X,T)$ is (topologically) rigid if there exists an increasing sequence $(n_i)_{i\in \N}$ in $\N$ such that $T^{n_i}x$ converges to $x$ as $i$ goes to infinity for every $x\in X$ ({\em i.e.} $T^{n_i}$ converges pointwisely to the identity). A topological dynamical system is uniformly rigid if  $\sup\limits_{x\in X} d(x,T^{n_i}x) \to 0$ as $i$ goes to infinity, {\em i.e.} $T^{n_i}$ converges uniformly to the identity map. It is clear that uniform rigidity implies rigidity but the converse is not true even for minimal systems \cite{GM,K}.
By the Lebesgue Dominated Convergence Theorem, if $(X,T)$ is rigid (topologically) then $(X,\mathcal{B}(X),\mu,T)$ is rigid (in the measurable setting) for any invariant measure $\mu$, where $\mathcal{B}(X)$ is the Borel $\sigma$-algebra. So, as one could expect, the topological rigidity property is a much stronger notion than the measurable one. However, we show that there is no a real difference from the measurable point of view. Our main result states that any ergodic rigid system can be topologically realized in a uniformly rigid system.

\medskip

Let $(X,\X, \mu, T)$ be an ergodic dynamical system. We say that $(\hat{X}, \mathcal{B}(\hat{X}), \hat{\mu}, \hat{T})$ is a {\em topological model}
(or just a {\em model}) for $(X,\X, \mu, T)$ if $(\hat{X}, \hat{T})$ is a
topological system, $\hat{\mu}$ is an invariant Borel probability measure on $\hat{X}$ and the systems
$(X,\X, \mu, T)$ and $(\hat{X}, \mathcal{B}(\hat{X}), \hat{\mu}, \hat{T})$ are measure
theoretically isomorphic. In this case, one also says that $(X,\X,\mu,T)$ can be (topologically) realized by $(\hat{X},\hat{T})$.

\begin{thm} \label{THM1}
	Let $(X,\X,\mu,T)$ be a non-periodic ergodic  invertible measure preserving system, rigid for the sequence $(n_i)_{i\in \N}$. Then, there exists a topological model $(\widehat{X},\widehat{T})$ for $(X,\X,\mu,T)$ which is uniformly rigid for a subsequence of $(n_i)_{i\in \N}$. Moreover, $(\widehat{X},\widehat{T})$ can be taken topologically weak mixing.
\end{thm}

Putting $\mathcal{A}$ to be the algebra of continuous functions on $\widehat{X}$ we deduce

\begin{cor}
	Let $(X,\X,\mu,T)$ be an ergodic measure preserving system, rigid for the sequence $(n_i)_{i\in \N}$. Then there exists a subsequence $(n_i')_{i\in \N}$ of $(n_i)_{i\in \N}$ and a separable subalgebra $\mathcal{A}\subset L^\infty(\mu)$ which is dense in $L^{2}(\mu)$ such that $\|f-f\circ T^{n_i'}\|_{\infty}\to 0$ for any $f\in \mathcal{A}$.
\end{cor}
This result is attributed to Weiss in \cite{GM} but the proof was not published.

\medskip
A sequence  $(n_i)_{i\in \N}$ is called {\it a rigidity sequence} if there exists a measure preserving system for which $(n_i)_{i\in \N}$ is a rigidity sequence.  Since in Theorem \ref{THM1} we get a subsequence of the original sequence, a natural question arise is the following:

\begin{prob}
	Give conditions for a sequence $(n_i)_{i\in \N}$ to be a uniform rigidity sequence for a non-periodic topologically weakly mixing dynamical system. Is there a sequence $(n_i)_{i\in \N}$ which is a rigidity sequence (in the measurable category) but it is not a uniform rigidity sequence (in the topological category)?
\end{prob}

\bigskip

\section{Preliminaries}

\subsection{Measurable and topological systems}
A {\em measure preserving system} is a 4-tuple $(X,\mathcal{X},\mu,T)$ where $(X,\mathcal{X},\mu)$ is a probability space and $T$ is a measurable measure-preserving transformation on $X$. { In this paper, we assume that $T$ is invertible and both $T$ and $T^{-1}$ are measure-preserving transformations.} It is {\em ergodic} if any invariant set has measure 0 or 1. For an ergodic system, either the space $X$ consists of a finite set of points on which $\mu$
is equidistributed, or the measure $\mu$ is atom-less. In the first case the system is called
{\em periodic}, and it is called {\em non-periodic} in the latter.

A {\em topological dynamical system} is a pair $(X,T)$ where $X$ is a compact metric space and $T\colon X\to X$ is a homeomorphism. It is said to be {\em transitive} when there is a point $x\in X$ whose orbit $\{T^n x: n\in \Z\}$ is dense in $X$. It is {\em minimal} if any point has a dense orbit. A topological dynamical system is {\em weakly mixing} if the Cartesian product system $(X\times X,T\times T)$ is transitive. This is equivalent to that for any  four non-empty open sets $A,B,C,D$, there exists $n\in \Z$ such that $A\cap T^{-n}B\neq \emptyset$ and $C\cap T^{-n}D\neq \emptyset$. A topological dynamical system is (strongly) mixing if for any two non-empty open sets $A,B$ there exists $M\in \N$ such that { for any $n\in \Z$ with $|n|\geq M$ one has that $A\cap T^{-n}B\neq \emptyset$.}

By the Krylov-Bogoliuvov Theorem, any topological dynamical system $(X,T)$ admits a non-empty convex set of invariant probability measures, which is denoted by $M(X,T)$. The extremal points of $M(X,T)$ are the ergodic measures.

A deep link between measure preserving systems and topological dynamical systems is the Jewett-Krieger Theorem \cite{J,Kr}, which asserts that any ergodic non-periodic measure preserving system is measurably isomorphic to a uniquely ergodic topological dynamical system $(X,T)$, meaning that $(X,T)$ possesses only one invariant measure (which is ergodic). Many generalization, in different contexts have been found for the Jewett-Krieger Theorem \cite{Le,W1} and very recently several applications have been given for the pointwise convergence of different ergodic averages \cite{DS,DS2,HSY2} and to build interesting examples in topological dynamics \cite{LSY}. All these recent results show that topological dynamical systems can help to understand purely ergodic problems.



\subsection{Partitions}

Let $(X,\X,\mu, T)$ be a measure preserving system system. A {\em partition} $\a$ of $X$ is a family of disjoint measurable subsets of $X$ whose union is $X$. Let $\a$ and $\b$ be two partitions of $(X,\X,\mu,T)$. We say that $\a$ {\em refines} $\b$, denoted by $\a\succ \b$ or $\b\prec\a$, if each element of $\b$ is a union of elements of $\a$. $\a \succ \b$ is equivalent to $\sigma(\b)\subseteq \sigma(\a)$, where $\sigma(\mathcal{A})$ is the $\sigma$-algebra generated by the family $\mathcal{A}$.

\medskip

Let $\a$ and $\b$ be two partitions. Their {\em join} is the partition $\a\vee\b=\{A\cap B:A\in \a, B\in \b\}$ and one can extend this definition naturally to a finite number of partitions. For $m\le n$, define
$$\a_{m}^n=\bigvee_{i=m}^nT^{-i}\a=T^{-m}\a\vee T^{-(m+1)}\a \vee\ldots\vee T^{-n}\a ,$$
where $T^{-i}\a=\{T^{-i}A: A\in \a\}$.

\subsection{Rohklin towers}
Let $(X,\mathcal{X},\mu,T)$ be an ergodic measure preserving system and let $A$ be a measurable set. If $N\in \N$ and the sets $A,TA,\ldots , T^{N-1}A$ are pairwise disjoint, the array
\[ \mathfrak{c}= \{A,TA,\ldots ,T^{N-1}A \} \]
is called a {\em column } or {\em Rohklin tower} with base $A$ and height $N$. We usually refer to the sets $T^{i}A$, $i=0,\ldots,N-1$ as the levels of the column. The levels $A$ and $T^{N-1}A$ are called {\em base} and {\em roof} respectively.

$$
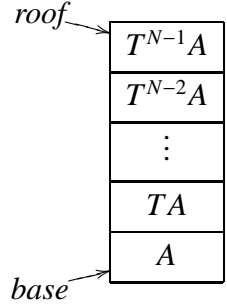
\xymatrix @W=3pc @H=1pc @R=0pc @*[F-] {
	T^{N-1}A \save+<-4pc,1pc>*\hbox{\it roof}
	\ar[]
	\restore \\
	{T^{N-2}A}
	\save*{}
	\restore \\
	{\vdots}
	\save*{}
	\restore \\
	TA
	\restore  \\ A\save+<-4pc,-1pc>*\hbox{\it base}
	\ar[]
	\restore}
$$

\medskip

A set $\mathfrak{t}$ is called a {\em tower} if it is a disjoint union of columns $$\mathfrak{c_i}=\{ A^i, TA^i,\ldots, T^{N_i-1}A^i \},\quad i=1,\ldots, l.$$ The union of the bases $\bigcup_{i=1}^l A^i$ is the {\em base} of $\mathfrak{t}$ and the union of the roofs $\bigcup_{i=1}^l T^{N_i-1}A^i$ is the {\em roof} of $\mathfrak{t}$.

\medskip

\subsection{Kakutani-Rokhlin towers}
For an ergodic system $(X,\X,\mu, T)$, let $B\in \X$ be a set with positive measure.
Then it is clear that \(\bigcup_{n\geq 0}T^nB=X\) (mod\ $\mu$). Define the {\em return time function} $r_B: B \rightarrow \N\cup\{\infty\}$ by
\[r_B(x)=\min \ \{n\geq 1:T^n x \in B\}\]
when this minimum is finite and $r_B(x)=\infty$ otherwise. Let $B_k=\{x\in B:r_B(x)=k\}$
and note that by Poincar\'{e}'s recurrence theorem $B_\infty$ is a null set. Let $\mathfrak{c}_k$
be the column $\{B_k,TB_k,\ldots,T^{k-1}B_k\}$. We call the tower
$$\mathfrak{t}=\mathfrak{t}(B)=\{\mathfrak{c}_k:k=1,2...\}$$ the {\em Kakutani tower over $B$}. If the Kakutani tower over B
has finitely many columns ({\em i.e.} the function $r_B$ is bounded) we say that $B$ has a {\em finite height}
and we call the Kakutani tower over $B$ a {\em Kakutani-Rokhlin tower} or {\em a K-R tower}.
The number $\max r_B$ is called the {\em height} of B or the {\em height} of the K-R tower.

\medskip

We will need the following useful lemma (see \cite{G,W, Weiss00} for a proof), which is a special case of the Alpern Lemma \cite{Alpern}.


\begin{lem}\label{KR tower}
	Let $(X,\mathcal{X},\mu,T)$ be a non-periodic ergodic system.
	For any positive integers $N_1,N_2$ with  $gcd(N_1,N_2)=1$, there exists a set $C$ of finite height such that the K-R tower \(\mathfrak{t}(C)\) satisfies range $r_C=\{N_1,N_2\}$.
\end{lem}

\subsection{Refining a tower according to a partition}
Let $(X,\X,\mu,T)$ be a measure preserving system. Let $\mathfrak{t}$ be a tower with columns $\{\mathfrak{c}_k:k\in K\}$
($K$ is finite or countable) and base $B=\bigcup_{k\in K}B_k\subseteq X$. Given a partition
(finite or countable) $\a$ of $X$, we define an equivalence relation on $B$ as
follows: $x\sim y$ iff $x$ and $y$ are in the same base $B_k$ and for every
$0\le j< N_k$, $T^jx$ and $T^jy$ are in the same elements of $\a$, {\em i.e.}
$x$ and $y$ have the same $(\a, N_k)$-name. Now we consider each
equivalence class $B_{k,{\bf a}}$, with {\bf a} an $(\a, N_k)$-name,
as a base of the column $\mathfrak{c}_{k, {\bf a}}=\{B_{k,{\bf a}}, TB_{k,{\bf a}},
\ldots,T^{N_k-1} B_{k,{\bf a}}\}$ and say that the resulting tower $\mathfrak{t}_\a=
\{\mathfrak{c}_{k,{\bf a}}: {\bf a}\in \a^{N_k}, k\in K\}$
is the {\em tower $\mathfrak{t}$ refined according to $\a$}. We usually refer to the columns of the refined tower as {\em pure columns}.


\subsection{Symbolic dynamics}

Let $\Sigma$ be a set. Let $\Omega={\Sigma}^{\Z}$ be the
set of all sequences $\omega=\ldots {\omega}_{-1}{\omega}_0{\omega}_1 \ldots=({\omega}_n)_{n\in \Z}$, ${\omega}_n \in
\Sigma$, $n \in \Z$, endowed with the product topology. The shift map $\sigma: \Omega \rightarrow \Omega$
is defined by $(\sigma \omega)_n = {\omega}_{n+1}$ for all $n \in \Z$. The
pair $(\Omega,\sigma)$ is called the {\em full shift} over $\Sigma$. Any subsystem (closed and invariant subset) of $(\Omega, \sigma)$
is called a {\em subshift}.

Each element of ${\Sigma}^{\ast}= \bigcup_{k \ge 1} {\Sigma}^k$ is called {\em a word} or {\em a block} (over $\Sigma$). If $A=a_1\ldots a_n$, we use $|A|=n$ to denote its length. If $\omega=\cdots \omega_{-1} \omega_0 \omega_1 \cdots \in \Omega$ and $a \le b \in \Z$, then
$\omega[a,b]=:\omega_{a} \omega_{a+1} \cdots \omega_{b}$ is a $(b-a+1)$-word occurring in
$\omega$ starting at place $a$ and ending at place $b$.
Similarly we define $A[a,b]$ when $A$ is a word. A word $A$ {\em appears} in the word $B$ if there are some $a\le b$ such that
$B[a,b]=A$.

For $n\in\N$ and words $A_1,\ldots, A_n$, we denote by $A_1\ldots A_n$ the concatenation of $A_1,\ldots, A_n$.
When $A_1=\ldots=A_n=A$ denote $A_1\ldots A_n$ by $A^n$. 
If $(X,\sigma)$ is a subshift , let $[i]=[i]_X=\{\omega\in X:\omega(0)=i\}$ for $i\in \Sigma$, and
$[A]=[A]_X=\{\omega\in X:\omega_0 \omega_1\cdots \omega_{(|A|-1)}=A\}$ for any word $A$.



\subsection{Symbolic representation}

Let $(X,\X,\mu,T)$ be an ergodic measure preserving system. Given a measurable function $f\colon X\to \Sigma\subseteq [0,1]$, one can define the itinerary homomorphism $f^{\infty}$ from  X to $\Omega:=[0,1]^{\Z}$ given by $f^{\infty}(x)=\omega$, where
\[\omega_n=f(T^nx).\]
The distribution of the stochastic process $(f^{\infty})_*(\mu)$ (defined by $(f^{\infty})_*(\mu)(A)=\mu((f^{\infty})^{-1}(A))$, for each Borel subset $A\subset[0,1]^{\Z}$), is denoted by $\rho(X,f)$ and we call it the {\em representation measure given by f} of $(X,T)$. When the system under consideration $(X,\X,\mu,T)$ is fixed, we just write $\rho$ instead of $\rho(X,f)$ for convenience.

Let
$$X_f={\rm supp}\big((f^{\infty})_*(\mu)\big)={\rm supp} (\rho).$$
Then we get a homomorphism $f^{\infty}\colon (X,\X,\mu,T)\rightarrow (X_f, \X_f, \rho, \sigma)$.
This homomorphism is called the {\em representation} of the process $(X,f)$.

A very important case is when we consider a finite partition $\a=\{A_j\}_{j\in \Sigma}$ (we assume $\mu(A_j)>0$ for all $j$). Here $\Sigma\subset [0,1]$ is a subset of real numbers (not necessarily integers). We think of the partition $\a$ as the function $f_{\a}$ defined as $f_{\a}(x)=j$ if $x\in A_j$. Equivalently, when $f$ has finitely many values $\{a_1,\ldots,a_k\}$ we can think of $f$ as the function given by the partition $\a=\{A_j\}_{j\in \Sigma}$ where $A_j=f^{-1}(a_j)$. Let $(X,\a)$ denote the representation $(X,f_{\a})$ and we call it the {\em symbolic representation} given by the partition $\a$.

This will not be a model for $(X,\X,\mu,T)$ unless $\bigvee_{i=-\infty}^\infty T^{-i}\a=\X$
modulo null sets.





\subsection{Copying names}  \label{Sec:PaintingNames}

An important way to produce partitions (equivalently, finite valued functions) is by copying or painting names on towers.


If $\mathfrak{c}=\{T^jB\}_{j=0}^{N-1}$ is a column and ${\bf a}\in \Sigma^N$ then
{\em copying the name ${\bf a}$ on the column $\mathfrak{c}$} means that on
$\bigcup_{j=0}^{N-1}T^jB$ we define a partition (may not be on the whole space)  by letting
\begin{equation*}
A_k=\bigcup\{T^jB:{\bf a}_j=k\}, \quad k\in \Sigma.
\end{equation*}
If there is a tower $\mathfrak{t}$ with $q$ columns $\mathfrak{c}_i=\{T^jB_i\}_{j=0}^{N_i-1}$, and
$q$ names ${\bf a}(i)\in \Sigma^{N_i}, i=1,\ldots,q$, then copying these names on $\mathfrak{t}$
means we copy each name ${\bf a}(i)$ on column $\mathfrak{c}_i$, {\em i.e.} we define a partition by
\begin{equation*}
A_k=\bigcup\{T^jB_i: {\bf a}(i)_j=k, i=1,\ldots, q\}, \quad k\in \Sigma.
\end{equation*}
These partitions can be extended to a partition $\a=\{A_{a_1}, \ldots, A_{a_l}\}$ of the whole space by
assigning, for example, the value $a_1$ to the rest of the space.

\section{Proof of Theorem \ref{THM1}}

In this section we prove Theorem \ref{THM1}. For the sake of clarity, we divide the proof into two steps. First, we prove that we can realize an ergodic rigid system in a uniformly rigid topological dynamical system and then we show how to add the (topologically) weakly mixing condition.

\medskip

Let $(X,\X,\mu,T)$ be an ergodic rigid system with rigidity sequence $(n_i)_{i\in \N}$. { We start considering a special topological model: we may assume, by \cite{Le}, that $(X,T)$ is a minimal (strongly) mixing subshift (in fact, since a rigid system has zero entropy, we may consider a subshift over two symbols \cite{W, Ho}, but we do not need this property).}

\subsection{Proof strategy}
First, it is worth noting that a model given by a finite partition does not fit to our purposes as the following remark shows:

\begin{rem}
	Let $(X,\sigma)$ be a non-periodic (equivalently infinite) subshift. Then $(X,\sigma)$ is not rigid.
\end{rem}

\begin{proof}
	It is well-known that infinite symbolic systems have always a forward asymptotic pair (see \cite{Aus} Chapter 1 for example), {\em i.e.} there exist $\omega, \omega'\in X$ such that $\omega_{0}\neq \omega'_{0}$ and $\omega_n=\omega'_n$ for all $n\geq 1$. If $(X,\sigma)$ is rigid for the sequence $(n_i)_{i\in \N}$ then  $\sigma^{n_i}\omega\to \omega$ and $\sigma^{n_i}\omega'\to \omega'$ which implies that $\omega_0=\omega_{n_i}=\omega'_{n_i}=\omega'_{0}$, a contradiction.
\end{proof}

The proof of Theorem \ref{THM1} relies in the idea of building a topological model for an ergodic system using itineraries of a given function. This idea was already used in \cite{Ho,W} to find special models for systems with zero entropy.  Let $f\colon X\to [0,1]$ be a measurable function. Recall that the itinerary function $f^{\infty}\colon X\to [0,1]^{\Z}$ is $$f^{\infty}(x)=(\ldots,f(T^{-2}x),f(T^{-1}x),f(x),f(Tx),f(T^2x),\ldots)$$ and that the topological system associated to $f$ is the support of the measure $(f^{\infty})_*(\mu)$ in $[0,1]^{\Z}$ endowed with the shift action.

The function $f\colon X\to [0,1]$ generates for $T$ if the $\sigma$-algebra generated by the functions $f\circ T^n$, $n\in \Z$ is all of $\X$ (mod null sets). This is equivalent to that there exists a set of full measure $A$ on which the itinerary function $f^{\infty}$ is injective (see \cite{Sh} Chapter 1 for a reference). Thus, when $f$ generates for $T$ we have that the itinerary function $f^{\infty}$ is an isomorphism between $(X,\X,\mu,T)$ and $(X_f, \X_f, \rho, \sigma)$.

The general strategy consists in finding a sequence of functions $(f_i)_{i\in \N}$, where $f_{i+1}$ and $f_i$ differ in a set of small measure so that there exists a pointwise  limit function $f$. Suitable properties to the functions $f_i$ are required so that the corresponding topological system associated to $f$ satisfies the properties we are looking for.

Each $f_i$ will generate for $T$, and we will guarantee that $f$ generate for $T$ by controlling the speed of convergence of $f_i$ to $f$. The $f_i$'s will be continuous and each one will take only finitely many values, so we may identify them with finite partitions $\a_i$ of $X$ into clopen sets, where $f_i: X\rightarrow \{a_1,a_2,\ldots, a_{m_i}\}\subseteq [0,1]$ and $\a_i=\{A_1,\ldots,A_{m_i}\}$ with $A_j=f^{-1}_i(a_j)$.

In our case, the condition we ask is that any function $f_i$ is close to be uniformly rigid. To do this we introduce the following definition.

\begin{defin}
	We say that $f\colon X \to [0,1]$ is {\em $\epsilon$-good at $n$} if
	\[ \|f-f\circ T^{n}\|_{\infty} < \epsilon. \]
	
\end{defin}
Here $\|\cdot \|_{\infty}$ stands for the essential supremum norm. Of course if $f$ is continuous this coincides with the supremum norm.

Let $(K_i)_{i\in \N}$ be a sequence of positive integer numbers such that $\sum _{i=1}^\infty \frac{1}{K_i} < \infty$.

Our goal is to build a sequence of generating and continuous functions $(f_i)_{i\in \N}$ and a subsequence $(n_i')_{i\in \N}$ of $(n_i)_{i\in \N}$ such that
\begin{equation} \label{EQ:ConditionRigidity}
\begin{split}
f_i \quad \text{ is } \quad \left( \sum_{l=j}^{i}\frac{1}{K_l} \right) \text{-good at } n'_j \text{ for any } j\leq i \\
\mu(\{f_i\neq f_{i+1}\})< r_i,
\end{split}
\end{equation}
where $r_i$ goes fast enough to 0 (for instance $r_i=2^{-i}$).  In this case, we say that the sequence {\em $(f_i)_{i\in \N}$ is good for the sequence $(K_i)_{i\in \N}$}.

We will also impose that the cardinality of the image of $f_{i+1}$ is strictly larger than the one of $f_i$.
This guarantees that the pointwise limit of $f_i$ is well defined and also generates for $T$. To see that $f=\lim f_i$ generates for $T$, remark that since the functions $f_i$ are generating, there exists a set of full measure $A$ where all $f_i^{\infty}$ are injective (see \cite{Sh} Chapter 1 for example). The Borel-Cantelli Lemma ensures that in a set of full measure $B$, $x\in B$ implies that $f^{\infty}(x)=f^{\infty}_i(x)$ for some $i\in \N$. So if $x,y \in A\cap B$ and $f^{\infty}(x)=f^{\infty}(y)$, then there exists $i,j$ such that $f^{\infty}_i(x)=f^{\infty}_{j}(y)$. We can assume $j>i$, since $i=j$ is not possible by the injectivity of $f^{\infty}_i$ in $A$. There is an open subset of $X$ where the value of $f_j$ is different from all values of $f_i$ (recall that the functions are continuous). The minimality of $(X,T)$ implies that for some $n$, $f_i(T^nx)\neq f_j(T^n y)$. This shows that $f^{\infty}$ is injective in a set of full measure and so $f$ generates for $T$.

\subsection{Some Facts}

Our proof is based on doing modifications of a tall enough tower. We modify these towers by taking averages between given portions of a subcolumn. We formalise this idea with the next definition.

{ Let $A=a_1\ldots a_n$ and $B=b_1\ldots b_n$ be two blocks and $\lambda\in \R$. Write $\lambda A=(\lambda a_1)\ldots (\lambda a_n)$ and $A\pm B=(a_1\pm b_1)\ldots (a_n\pm b_n)$.}

\begin{defin}
	Let $A=a_1\ldots a_n \in [0,1]^n$, $B=b_1\ldots b_n\in [0,1]^n$ and $K\in \N$. We say that $C=c_1\ldots c_{(K+1)n}$ is a {\em transition} from $A$ to $B$ in $K$-steps if $C$ is the concatenation of the blocks $A+\frac{j}{K}(B-A)$ for $j=0,\ldots,K$.
\end{defin}

\begin{rem}
	$A$ and $B$ represent two given subcolumns of length $n$ and $C$ represents a subcolumn of length $(K+1)n$ where the first and last $n$ levels are $A$ and $B$ respectively.
\end{rem}

\begin{lem} \label{Lem:goodTransition}
	Let $A=a_1\ldots a_n,$ $B=b_1 \ldots b_n \in [0,1]^{n}$ and let $C=c_1\ldots c_{(K+1)n} \in [0,1]^{(K+1)n}$ be the transition from $A$ to $B$ in $K$-steps. Then, for any $l=1,\ldots,Kn$ we have that \[ |c_{l}-c_{l+n}|\leq \frac{1}{K} \]
	
\end{lem}

\begin{rem}
	This lemma shows that if $K$ is big enough then we have a ``smooth'' $K$-step transition between two  blocks of the same lengths, which will be useful to ensure rigidity.
\end{rem}

\begin{proof}[Proof of Lemma \ref{Lem:goodTransition}]
	We have that there exist $j\leq K-1$ and $1\leq p\leq n$ such that $c_{l}=\frac{K-j}{K}a_p+\frac{j}{K}b_p$ and $c_{l+n}=\frac{K-j-1}{K}a_p+\frac{j+1}{K}b_p$. Thus \[ c_{l}-c_{l+n}= \frac{a_p-b_p}{K} \] and the result follows.
\end{proof}

The next lemma shows that if two blocks have a similar top and bottom, then when performing a transition between them, the top of a block and the bottom of the consecutive one have a ``smooth'' transition. This condition will be useful in order to get the first property in \eqref{EQ:ConditionRigidity}.

\begin{lem} \label{Lem:Transition}
	Let $A=a_1 \ldots a_n$,  $B=b_1\ldots b_n\in [0,1]^{n}$ and let $C=c_1\ldots c_{(K+1)n} \in [0,1]^{(K+1)n}$ be the transition from $A$ to $B$ in $K$-steps. Let $n/2\geq p\geq l \geq 0$. If  $\vert a_{n-p+l}-a_{l} \vert \leq \delta $ and $\vert b_{n-p+l}-b_{l} \vert\leq \delta$ (i.e. $A$ and $B$ have similar top and bottom)
	then for every $j=0,\ldots K-1$ we have that $\vert c_{jn+n-p+l}- c_{(j+1)n+l} \vert \leq \delta+\frac{1}{K}$.
\end{lem}

\begin{rem}
	We think of the term $c_{jn+n-p+l}$ as some level close to the top of the block $A+\frac{j}{K}(B-A)$ while  the term $c_{(j+1)n+l}$ is a level close to the bottom of the block $A+\frac{j+1}{K}(B-A)$.  
\end{rem}

\begin{proof}
	By definition we have that
	\begin{align*}
	c_{jn+n-p+l}- c_{(j+1)n+l}  & = a_{n-p+l}+\frac{j}{K}(b_{n-p+l}-a_{n-p+l})-a_{l}-\frac{j+1}{K}(b_{l}-a_{l}) \\
	& = \frac{K-j}{K}(a_{n-p+l}-a_{l})+\frac{j}{K}(b_{n-p+l}-b_{l})-\frac{b_l-a_l }{K}
	\end{align*}
	and the result follows.
\end{proof}

\begin{lem} \label{Lem:RigidPowers}
	Let $(X,\X, \mu,T)$ be a measure preserving rigid system. Then for each $k\in \N$, $(X,\X,\mu,T^k)$ is also rigid.
\end{lem}

\begin{proof}
	Let $(n_i)_{i\in \N}$ be a rigidity sequence for $T$ and let $f\in L^{2}(\mu)$. We have that $\|f-f\circ T^{n_ik}\|_2\leq \sum_{j=0}^{k-1} \|f\circ T^{n_ij}-f\circ T^{n_i(j+1)}\|_2=k\| f - f\circ T^{n_i}\|_2 \to 0$. We conclude that $(n_i)_{i\in \N}$ is also a rigidity sequence for $T^k$.
\end{proof}

\subsection{Proof of Theorem \ref{THM1}: Getting a uniformly rigid model}



We now proceed to prove Theorem \ref{THM1}. {Recall that we assume that $(X,T)$ is a minimal (strongly) mixing subshift and we consider a sequence of positive number $(r_i)_{i\in \N}$ converging fast enough to 0 (for instance $r_i=2^{-i}$).
	
	Let $(K_i)_{i\in \N}$ be an increasing sequence of positive integers such that $\sum \frac{1}{K_i} < \infty$. For simplicity we assume $K_0=1$.
	We construct the sequence of functions $(f_i)_{i\in \N}$ good for $(K_i)_{i\in \N}$ inductively.
	
	Let $\a_0=\{A_1,\ldots,A_{m_0}\}$ be a clopen generator for $T$, and $a_1,\ldots,a_{m_0}$ be real numbers in $[0,1]$. Let $f_0: X\rightarrow \{a_1,a_2,\ldots, a_{m_0}\}\subseteq [0,1]$ such that $A_j=f^{-1}_i(a_j)$ for $1\le j\le m_0$.} It is a continuous function and since $K_0=1$, we have that $f_0$ satisfies trivially the properties we require for any $n_0'\in (n_k)_{k\in \N}$ (we consider values in $[0,1]$).
To illustrate our method and make the proof clearer we show how to obtain $f_1$ from $f_0$.

\medskip

\noindent {\bf Step $1$:}
Let $\alpha_0$ denote the partition associated to the different values of $f_0$ ({\em i.e.} $\alpha_0$ is the  canonical partition at the origin). Consider the integer $K_1$ and the positive number $r_1$. Since $f_0$ has finitely many values, there exists a constant $c_0>0$ such that $|f_0(x)-f_0(y)|\leq c_0$ implies $f_0(x)=f_0(y)$.

\medskip

For $k\in \N$, consider the set
\[ A_{0,k}\coloneqq \left\{ x\in X : |f_0(x)-f_0(T^{ln_k}x)|> c_0 ~~\text{ for some } l\in [1,2K_1]\cap \N \right \}. \]
Since, by Lemma \ref{Lem:RigidPowers}, the transformations $T,T^2,\ldots, T^{2K_1}$ are rigid for $(n_k)_{k\in \N}$, the measure of $A_{0,k}$ goes to 0 as $k$ goes to infinity. By our choice of $c_0$, the condition $x\in A_{0,k}^c$ implies that $f_0(x)=f_0(T^{n_k}x)=\cdots=f_0(T^{2K_{1}n_k}x)$.

We pick $n_{k_1}$ such that the measure of $A_{0,k_1}$ is smaller than $\frac{r_1}{4K_1}$ and we put $A_0=A_{0,k_1}$ and $n_1'=n_{k_1}$.

\medskip

We can use Lemma \ref{KR tower} to build a large Kakutani-Rokhlin tower of heights $H_1$ and $H_1+1$ (and with a clopen base). We then refine this column according to the $\alpha_0$-names.  We can assume that $H_1$ has the form $ 2K_1n'_1N_1+n'_1$, where $1/N_1\leq r_1/6$. We can subdivide every pure column into
$N_1$ subcolumns of length $2K_1n'_1$, starting from the bottom to the top. We call these {\em principal} subcolumns. The remaining $n_1'$ levels are called the {\em top}. For convenience, for those columns whose height is $H_1+1$ we add the top level to the top (so the top has $n'_1$ or $n'_1+1$ levels). Similarly, the first $n'_1$ levels are the {\em bottom} of the column (see Figure \ref{Fig:Towers}).

\begin{figure}[H]
	$$
	\xymatrix @W=6pc @H=1pc @R=0pc @*[F-] {
		{\hbox{$n_1'$ levels} }
		\save+<-6pc,1pc>*\hbox{\it top}
		\ar[]
		\restore
		\save+<3.5pc,-1.8pc>*{}
		\ar @{-} `r[dd] `[dddd]^{2K_1n_1' \text{levels}}
		\restore
		\save+<0 pc,3pc>*\hbox{\it $H_1$ column}
		\restore  \\
		\hbox{ $n_1'$ levels }
		\save*{}
		\restore \\
		{\vdots}
		\save*{}
		\restore \\
		\hbox{ $n_1'$ levels } \\
		\text{$n_1'$ levels}\\
		{\vdots}\\
		{\vdots}
		\save+<3.5pc,-1.8pc>*{}
		\ar @{-} `r[dd] `[dddd]^{2K_1n_1' \text{levels}}
		\restore \\
		\hbox{ $n_1'$ levels }
		\save*{}
		\restore \\
		{\vdots}
		\save*{}
		\restore \\
		{\hbox{ $n_1'$ levels }}
		\\ \hbox{ $n_1'$ levels }
		\save+<-6pc,-1pc>*\hbox{\it bottom}
		\ar[]
		\restore 
	}
	\xymatrix @W=6pc @H=1pc @R=0pc @*[F-] {
		\hbox{ $n_1'+1$ levels } \save+<-6pc,1pc>*\hbox{\it top}
		\ar[]
		\restore
		\save+<3.5pc,-1.8pc>*{}
		\ar @{-} `r[dd] `[dddd]^{2K_1n_1' \text{levels}}
		\restore
		\save+<0 pc,3pc>*\hbox{\it $H_1+1$ column}
		\restore\\
		\hbox{ $n_1'$ levels }
		\save*{}
		\restore \\
		{\vdots}
		\save*{}
		\restore \\
		\hbox{ $n_1'$ levels } \\
		\text{$n_1'$ levels}\\
		{\vdots}\\
		{\vdots}
		\save+<3.5pc,-1.8pc>*{}
		\ar @{-} `r[dd] `[dddd]^{2K_1n_1' \text{levels}}
		\restore\\
		\hbox{ $n_1'$ levels }
		\save*{}
		\restore \\
		{\vdots}
		\save*{}
		\restore \\
		{\hbox{ $n_1'$ levels }}
		\\ \hbox{ $n_1'$ levels }
		\save+<-6pc,-1pc>*\hbox{\it bottom}
		\ar[]
		\restore
	}
	$$    
	\caption{Principal subcolumns and top of a tower}
	\label{Fig:Towers}
\end{figure}
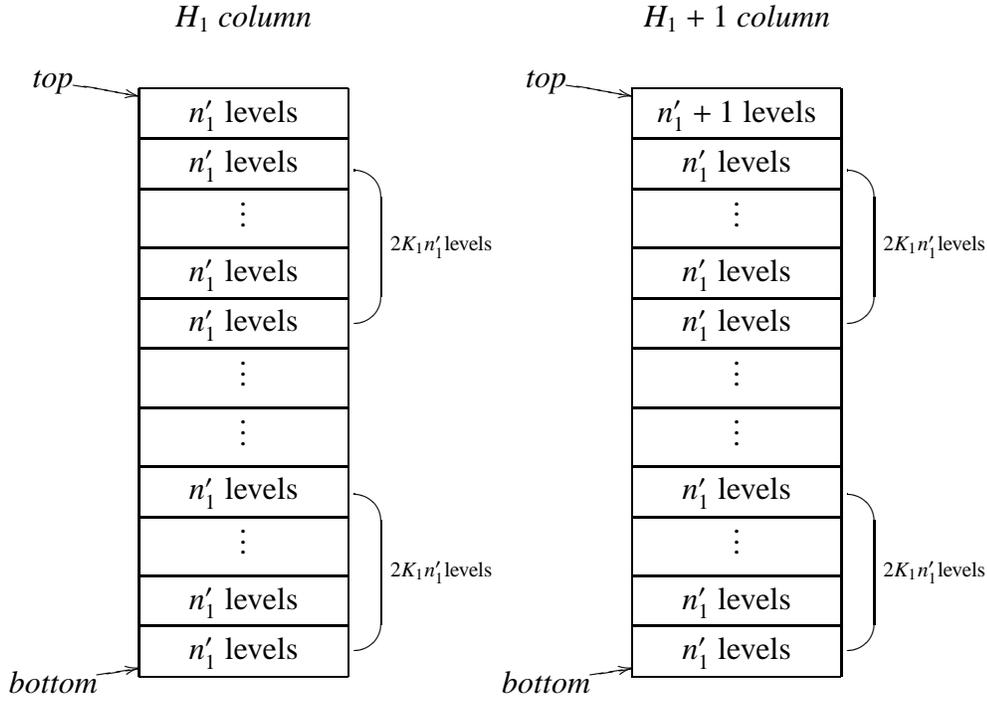

Our aim is to modify $f_0$ to $f_1$ such that $|f_1(x)-f_1(T^{n_1'}x)|< \frac{1}{K_1}$ for every $x\in X$. Translated to columns, this means that the difference of levels at distance $n_1'$ is smaller than $\frac{1}{K_1}$. Since $c_0$ is small enough, we have that in many cases two such levels are equal, but there is a small portion where this does not happen. We fix this problem allowing to the levels to take more values between $0$ and $1$. Now we explain how to do this. Let us consider two consecutive principal subcolumns and consider the first $n_1'$ levels of each one of them. We remark that if one level is in $A_0^c$ (meaning that the corresponding set of this level is a subset of $A_0^c$) , then it is constant in the $ln'_1$-levels above it for $l=1,\ldots,2K_1$. Indeed, this property characterizes belonging to $A_0$: a level who is in $A_0$ will change its value in some of the levels $ln'_1$, $l=1,\ldots,2K_1$ above it. We correct this values as follows:

\medskip

\noindent {\bf Step $1$-I: Modification of the top and the bottom. }We change the values of the top and the bottom of any pure column putting 0's, {\it i.e.} we paint (please recall Section \ref{Sec:PaintingNames}) the bottom and top with the 0 symbol on each level. This step is to ensure that the transition from one pure column to another one is $1/K_1$ good at $n'_1$. We may lose the property that $f_0$ is a generating function, but we fix this later in the end of the next step.

\medskip

\noindent {\bf Step $1$-II: Modification inside a pure column.} Consider two consecutive principal subcolumns and look at the first $n_1'$ levels of the first one and the first $n_1'$ levels of the second one. Perform a transition in $2K_1$-steps between these two subcolumns. Lemma \ref{Lem:goodTransition} ensures that all levels of the first principal subcolumn become $1/K_1$-good at $n_1'$.

This of course may change the $2K_1-1$ remaining levels of the first principal subcolumn but in fact we see that not many of them are modified: among the first $n_1'$ levels, those who belong to $A_0^c$ remain unchanged in their $n_1'$ translations. Recall that this follows from the fact that if  $x\in A_0^c$ then
$|f_0(x)-f_0(T^{ln_1'}x)|\leq c_0$ for all $l=1,\ldots,2K_1$ which implies that $f_0(x)=f_0(T^{n_1'}x)=\cdots=f_0(T^{2K_1n_1'}x)$.

In the other hand, we remark that for any level in $A_0$ we change at most $2K_1-1$ levels, so the quantity of levels we have changed in the first principal subcolumn is at most
\[ (2K_1-1)\#(\text{Levels in } A_0 \text{ in the first } n'_1 \text{ levels}).\]

\medskip

We repeat doing this process for all principal subcolumns, remarking that in the last one we perform the transition using the top (which has zeros). Therefore, any level is $1/K_1$-good for $n_1'$. It remains to show that we have modified $f_0$ in a small set.

For the first and last principal subcolumn and the top $n_1'$ levels we may change all levels, which are not more than  $4K_1n_1'+n_1'$. For any other principal subcolumn we do not change more than $$(2K_1-1)\#(\text{Levels in }A_0 \text{ in the first } n_1' \text{ levels}).$$ Therefore, in any pure column we change at most
\[ (4K_1+1)n_1'+(2K_1-1)\#(\text{Levels in } A_0) \]
(here the quantity of levels in $A_0$ is an upper bound for the quantity of levels we may find in the first $n_1'$ levels of the principal subcolumns).

Therefore, we modified any pure column in a proportion at most
\[ \frac{(4K_1+1)n_1'+1 +(2K_1-1)\#(\text{Levels in }A_0)}{N_12K_1n_1'+n_1'} \]
and therefore we have changed $f_0$ in a set of measure smaller than
\[\frac{3}{N_1}+ (2K_1-1)\mu(A_0) \]
and this set has measure smaller than $r_1$ by our assumptions. Since all levels are clopen sets, we have built a continuous function $f_1$ (with finitely many values) whose associated partition $\alpha_1$ is close to $\alpha_0$ in the partition metric. The function $f_1$ is $1/K_1$-good at $n'_1$ and $1/K_0+1/K_1$-good at $n_0'$ (this last condition holds trivially in this case).

We then make sure that all pure columns are different, modifying the first level of each one by an amount much smaller than all the constants involved, {\it i.e.} we paint (recall Section \ref{Sec:PaintingNames}) the first level of each pure column with a different, but very small value. Recall that the definition of being good involve a strict inequality, so we have enough freedom to achieve this without break the being good property.

By doing all pure columns different, we get that the sets defined by $\alpha_0$ names of length $H_1$ are the union of different $\alpha_1$ names of length $H_1$, which implies that $\alpha_1$ is also a generating partition.

We remark that we have to perform the modification in the order we gave. We need to perform the transitions of blocks after the modifications of the top and bottom, in order to correct the lack of rigidity we may have introduced.

\bigskip

\noindent {\bf Step $i+1$:}
The general case, {\em i.e.} how to obtain $f_{i+1}$ from $f_i$, is similar, but we have to be careful that when trying to fix being $1/K_{i+1}$-good at $n_{i+1}'$ we do not spoil the previous good conditions (at this step being topologically mixing will help us).

\medskip

Suppose we are given $f_i$ and $n_1',\ldots,n_i'$ such that $f_i$ is $\left( \sum_{l=j}^{i}\frac{1}{K_l} \right)$-good at $n_j'$ for $j\leq i$. We now show how to find $n_{i+1}'$ and build $f_{i+1}$ satisfying the corresponding properties.

Since $f_i$ takes finitely many values we have that there exists $c_i>0$ such that $\vert f_i(x)-f_i(y)\vert \leq c_i$ implies that $f_i(x)=f_i(y)$.

Since $(X,T)$ is topologically mixing, there exists $L_{i}\geq n_i'$ such that any couple itineraries of length $n_i'$ can be joined by an itinerary of any length greater or equal than $L_i$.

\medskip

Consider the set
\[A_{i,k}=\left \{x: |f_i(x)-f_i(T^{ln_k}x)|> c_i \text{ for some } l=1,\ldots,2K_{i+1}  \right \}\]
Since $T,T^2,\ldots,T^{2K_{i+1}}$ are rigid , we have that for big enough $k_{i+1}$ the measure of $A_{i,k_{i+1}}$ is smaller than $\frac{r_{i+1}}{6K_{i+1}}$ and of course we can also require that $\frac{2L_i}{n_{k_{i+1}}}\leq \frac{r_i}{3}$.

Put $A_i=A_{i,k_{i+1}}$ and $n_{i+1}'=n_{k_{i+1}}$ as above. We remark that $x\in A_i^c$ implies that the values $f_i(x),f_i(T^{n'_{i+1}}x),\ldots,f_i(T^{2K_{i+1}n'_{i+1}}x)$ are all equal.

\medskip

We then use Lemma \ref{KR tower} to construct a tower with heights $H_{i+1}$ and $H_{i+1}+1$ and we can assume that $H_{i+1}=N_{i+1}2K_{i+1}n'_{i+1}+n'_{i+1}$, where $1/N_{i+1}\leq r_{i+1}/9$. Similarly as was done in the first step, we subdivide every pure column into
$N_{i+1}$ subcolumns of length $2K_{i+1}n'_{i+1}$, starting from the bottom to the top and we call them {\em principal} subcolumns. The remaining $n_{i+1}'$ levels are called the {\em top}. Again, for those columns whose height is $H_{i+1}+1$ we add the top level to the top. The first $n_{i+1}'$ levels are the {\em bottom} of the column.

Refine the columns accordingly to the names given by the partition $\alpha_i$ (the partition associated to the function $f_i$). Pick a pure column and modify it accordingly to the following steps:


\medskip

\noindent {\bf Step $(i+1)$-I: Modification of the bottom and the top.}

When we are close to the top a column, we do not know where the point will lie after $n_{i+1}'$ levels, so we will modify the bottoms and the tops of the columns so that this transitions satisfy the good conditions. To achieve this, we first modify the top and bottom of any pure column by putting 0's, {\it i.e.} we paint those levels with the symbol 0.

\medskip

\noindent{\bf Step $(i+1)$-II: Guarantee not spoil anything.}

We may continue similarly as in Step 1-II, {\it i.e.}  performing transitions between blocks. Unfortunately this does not suffices since by doing this we may break the conditions of being good for the previous steps. More precisely, the function $f_i$ satisfies $|f_i(x)-f(T^{n_i'}x)|\leq \left( \sum_{l=j}^{i}\frac{1}{K_l} \right)$-good at $n'_j$ for any $j\leq i$, but if we perform transitions we may lose this property, especially in the levels close to the bottom and top of the blocks we concatenate. In order to keep this property when performing transitions we need to ensure conditions so that Lemma \ref{Lem:Transition} can be applied. To guarantee such conditions we make use of mixing and we proceed as follows.

Pick a pure column and consider a principal subcolumn (different from the one at the bottom, whose $n_{i+1}'$ first levels are modified in {Step $(i+1)$-I}). Let $B=a_1\ldots a_{n_{i+1}'}$ be the block in $[0,1]^{n'_{i+1}}$ corresponding to the values of its firsts $n_{i+1}'$ levels. Let $B_1=a_1\ldots a_{n'_{i}}$ and $B_2=a_{n'_{i+1}-n'_{i}-L_i+1}\ldots a_{n'_{i+1}-L_{i}} \in [0,1]^{n'_i}$. Since we assume that $(X,T)$ is topologically mixing, we can find $B_3\in [0,1]^{L_i}$ such that $B_2B_3B_1$ is a valid itinerary of $f_i$. We then replace the top $L_i$ levels of $B$ by $B_3$ and we get the block $B'$. Since $B_2B_3B_1$ is a valid itinerary for $f_i$ we have that:
\[ \vert B'_{n'_{i+1}-n'_{j}+k}-B'_{k} \vert \leq \left(\sum_{l=j}^{i}\frac{1}{K_l} \right) \text{ for any } j\leq i  \text{ and any } k\leq n_j' \]

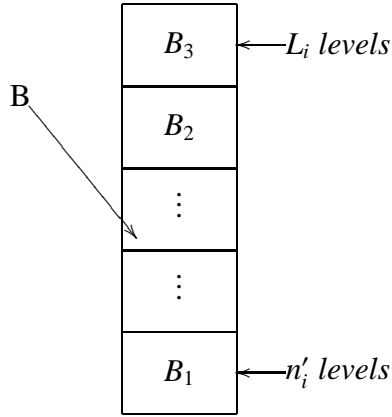
$$
\xymatrix @W=3pc @H=2pc @R=0pc @*[F-] {
	B_3 
	\save+<5pc,0pc>*\hbox{\it $L_i$ levels}
	\ar[]
	\restore
	\\
	B_2
	\save+<-5pc,1pc>*\hbox{B}
	\ar[dd]
	\restore  \\
	{\vdots}
	\\
	\vdots
	\\ 
	B_1
	\save+<5pc,0pc>*\hbox{\it $n_i'$ levels}
	\ar[]
	\restore
}
$$

\medskip

\noindent {\bf Step $(i+1)$-III: Modification inside a pure column.}
We are now ready to perform transitions.

Consider two consecutive principal subcolumns modified accordingly to { Step $(i+1)$-I and Step $(i+1)$-II} and perform a transition between the first $n_{i+1}'$ levels of these subcolumns. We recall that a level among the firsts $n_{i+1}'-L_{i}$ levels of a principal subcolumn (so it is not modified in { Steps $(i+1)$-I and $(i+1)$-II}) is in $A_i^c$ if and only if is constant in the $ln_{i+1}'$-levels above it for $l=1,\ldots,2K_{i+1}$. This means that the transition will not change the values of these levels. Lemma \ref{Lem:goodTransition} guarantees the precision $1/K_{i+1}$ we are looking for. The modifications we made in { Step $(i+1)$-II} and Lemma \ref{Lem:Transition} also ensures that the properties for $j\leq i+1$ are also respected (here we add some error term, given by $1/K_{i+1}$ but this value is small since we assume that the series is convergent.)

Again we modify the first level of each pure column in a small quantity such that all pure columns are different. The small quantity is chosen in order to keep the good properties of $f_i$ (which is defined by a strict inequality).

It remains to show that we have changed $f_i$ in a set of small measure. For any principal subcolumn (different from the ones at the bottom and top), we change at most
\[(2K_{i+1}-1)(L_i+\# (\text{ levels in } A_i \text{ among the firsts } n'_{i+1} \text{ levels }))   \]
levels. We may change all levels from the first and last principal subcolumns and the top ($n_{i+1}$ or $n_{i+1}+1)$ levels. Therefore, in a pure column the number of levels we change is at most

\[ 4K_{i+1}(n'_{i+1}+1) +1  +(2K_{i+1}-1)(N_{i+1}L_i+\# (\text{ levels in } A_i )) \]
and thus we have modified any pure column in a proportion smaller then
\[ \frac{4K_{i+1}n'_{i+1} + 1+ (2K_{i+1}-1)(N_{i+1}L_i+\# (\text{ levels in } A_i ))}{2N_{i+1}K_{i+1}n'_{i+1}+n'_{i+1}}. \]

From here we deduce that the set we modified has measure at most

\[ \frac{3}{N_{i+1}}+ \frac{2L_i}{n'_{i+1}} + (2K_{i+1}-1)\mu(A_i) \]
and this value is smaller then $r_{i+1}$ by our assumptions.
So, we have built $f_{i+1}$ which is continuous, generates for $T$ and $\mu(\{f_{i+1}\neq f_i\})<r_{i+1}$.

\medskip

We now consider the function $f$, the pointwise limit of the sequence $(f_i)_{i\in \N}$.

\medskip

\noindent {\bf Claim:}\quad $\| f- f\circ T^{n_i'}\|_{\infty} \to 0$ as $i$ goes to infinity.

\medskip

Let $X'$ be a set of full measure where $f_i$ converges to $f$. Let $\epsilon>0$ and let $j\in \N$ such that $\sum_{i\geq j}\frac{1}{K_i}\leq \epsilon/3$. Let $x\in X'$ and $i\geq j$. We can find $\bar{i}\geq i$ such that $|f_{\bar{i}}(x)-f(x)|\leq \epsilon/3 $ and $|f_{\bar{i}}(T^{n_i'}x)-f(T^{n_i}x)|\leq \epsilon/3$. Then, using that $f_{\bar{i}}$ is $\sum_{j=i}^{\bar{i}}\frac{1}{K_j}$-good for $n_i'$ we get that
\begin{align*}
|f(x)-f(T^{n_i'}x)| & \leq |f(x)-f_{\bar{i}}(x)|+|f_{\bar{i}}(x)-f_{\bar{i}}(T^{n_i'}x)|+|f_{\bar{i}}(T^{n_i'}x)-f(T^{n_i'}x)| \\
& \leq \epsilon/3+ \sum_{j=i}^{\bar{i}}\frac{1}{K_j}+\epsilon/3\leq \epsilon.
\end{align*}
Since $x$ and $i\geq j$ are arbitrary we get the conclusion.

\medskip

Now it remains to prove:

\medskip

\noindent {\bf Claim:}\quad The corresponding model $(X_f,\sigma)=(\text{supp} f^{\infty}\mu,\sigma)$ is uniformly rigid for $(n_i')_{i\in \N}$.

\medskip

Let $\epsilon>0$ and let $M\in \N$ such that if two sequences $\omega,\omega' \in [0,1]^{\Z}$ satisfy $|\omega_l-\omega_l'|\leq \epsilon/8$ for any $|l|\leq M$ then $d(\omega,\omega')\leq \epsilon/4$, { where $d$ is a metric on $X_f$.} Let $j$ such that $\|f-f\circ T^{n_{i}'}\|_\infty\leq \epsilon/2$ for any $i\geq j$.
Let $\omega$ be an arbitrary point in $Y$ and $i\geq j$. We can pick $x$ such that $\omega'=f^{\infty}(x)$ satisfy $|\omega_l-\omega_l'|\leq \epsilon/4$ for any $|l|\leq M+n_i'$ and such that  $|\omega'_{n_i'+p}-\omega'_p|\leq \epsilon/2$ for any $p\in \Z$. We deduce that
\begin{align*}
d(\sigma^{n_i'}\omega,\omega)\leq & d(\sigma^{n_i'}\omega,\sigma^{n_i'}\omega')+d(\sigma^{n_i'}\omega',\omega')+d(\omega,\omega') \\
\leq & \frac{\epsilon}{4} + \frac{\epsilon}{2} + \frac{\epsilon}{4}=\epsilon .
\end{align*}
Since this holds for any $i\geq j$ and $\omega\in Y$, we have that $(Y,\sigma)$ is uniformly rigid with rigidity sequence $(n_i')_{i\in \N}$.

\subsection{Proof of Theorem \ref{THM1}: Adding the weakly mixing condition}

We modify our construction in previous section such that the resulting system is (topologically) weakly mixing. Notice that though we assume that $(X,T)$ is mixing, this can not guarantee that $(X_f,\sigma)$ is weakly mixing.

{
	To make the representation $(X_f,\sigma)$ weakly mixing, one need to add the following condition: for all non-empty open sets $A,B,C,D$ there exists $n$ such that $\sigma^n A\cap B\neq \emptyset $ and $\sigma^nC\cap D\neq\emptyset$. This can be guaranteed by the following property of $\a$ (recall that $\a$ is the partition corresponding to $f$):
	For each $m\geq 0$ and $E_1,F_1,E_2,F_2\in \bigvee_{j=0}^{m-1}T^{-j} \a$, there is some $s$ such that
	$$\mu\times\mu((T\times T)^{s}(E_1\times F_1)\cap(E_2\times F_2)>0.$$
	To this aim, we need add similar property in each $\a_i$.
	The strategy in this section consist in incorporating this property gradually by modifying the bottom of a single pure column at each step (the ones described in the previous section) in such a way that we keep the rigidity property.}

\medskip

Now we give the details. Let $\{\a_i\}_{i=0}^\infty$ be the partitions in the previous section.

\begin{lem}\label{lem4.1}
	One can add the following properties in the partitions $\{\a_i\}_{i=0}^\infty$: there are a sequence of positive integers $\{s_i\}_{i=0}^{\infty}$, two sequences of positive
	numbers $\{r_i\}_{i=0}^{\infty},\{e_i\}_{i=0}^{\infty}$, with $r_{i+1}<\min\Big\{\cfrac{r_i}{2},\cfrac{e_i^2}{4i}\Big\}$, such that for all $i\ge 1$
	\begin{enumerate}
		\item $d(\a_{i},\a_{i+1})=\mu(\{f_i\neq f_{i+1}\})<r_{i+1}$.
		\item Let $\bigvee_{j=0}^{i-1}T^{-j}\a_{i}=\{U_1^i,\ldots,U^i_{\eta_i}\}$ with $U_j^i$ being nontrivial. Then there is a subset $\{U_1^{i+1},\ldots,U^{i+1}_{\eta_i}\}
		\subset\bigvee_{j=0}^{i-1}T^{-j} \a_{i+1}$ such that the $\a_{i}$-name of $U^i_h$ and the $\a_{i+1}$-name of $U^{i+1}_h$ are the same, $\forall 1\le h\le \eta_i$.
		\item for all $E_1,F_1,E_2,F_2 \in \{U_1^{i+1},\ldots,U^{i+1}_{\eta_i}\}$ as in (2), one has
		that $$\mu\times\mu((T\times T)^{s_{i+1}}(E_1\times F_1)\cap(E_2\times F_2))\geq e_{i+1}^2>0.$$
	\end{enumerate}
	
\end{lem}

\begin{proof}
	Assume that inductively we have constructed partitions $\{\a_i\}_{i=0}^n$, a sequences of positive integers $\{s_i\}_{i=0}^{n}$, two sequences of positive
	numbers $\{r_i\}_{i=0}^{n},\{e_i\}_{i=0}^{n}$, with $r_{i+1}<\min\Big\{\cfrac{r_i}{2},\cfrac{e_i^2}{4i}\Big\}$ for each $i\le n-1$.
	Let $\a_i=\{A_1^i,A_2^i,\ldots,A_{m_i}^i\}$ for $1\le i \le n $. Let $f_i\colon X\rightarrow \{a_1,a_2,\ldots, a_{m_i}\}\subseteq [0,1]$ such that $A^i_j=f^{-1}_i(a_j)$.
	
	\medskip
	
	The sequence $\{\a_i\}_{i=0}^n$ satisfies the following properties:
	for each $0\le i\le n-1$
	
	\noindent $(1)_i$: {\em
		We have $d(\a_{i},\a_{i+1})=\mu(\{f_i\neq f_{i+1}\})<r_{i+1}$.
		Let $\bigvee_{j=0}^{i-1}T^{-j}\a_{i}=\{U_1^i,\ldots,U^i_{\eta_i}\}$ with $U_j^i$ being nontrivial. Then there is a subset $\{U_1^{i+1},\ldots,U^{i+1}_{\eta_i}\}
		\subset\bigvee_{j=0}^{i-1}T^{-j} \a_{i+1}$ such that the $\a_{i}$-name of $U^i_h$ and the $\a_{i+1}$-name of $U^{i+1}_h$ are the same, $\forall 1\le h\le \eta_i$. Moreover,
		for all $E_1,F_1,E_2,F_2 \in \{U_1^{i+1},\ldots,U^{i+1}_{\eta_i}\}$, one has
		that $$\mu\times\mu((T\times T)^{s_{i+1}}(E_1\times F_1)\cap(E_2\times F_2))\geq e_{i+1}^2>0.$$
		
	}
	
	
	Now we make the induction for the $(1)_n$ case. First we need to define a word $\omega_{n}$
	which contains all pairs of names of non-trivial elements in $\bigvee_{i=0}^{n-1}T^{-i}\a_n$. We do it as follows.
	
	Let $\bigvee_{j=0}^{n-1}T^{-j}\a_{n}=\{U_1^n,\ldots,U^n_{\eta_n}\}$ with $U_i^n$ being nontrivial. Let $B_t$ be the name of $U^n_t$ for each $1\le t\le \eta_n$. Then  $W_{n}=\{B_1,B_2,\ldots,B_{\eta_n}\}\subset \{a_1,a_2,\ldots, a_{m_n}\}^{n}$ is the
	set of all names of nontrivial elements of $\bigvee_{i=0}^{n-1}T^{-i}\a_n$.
	Since $(X,T)$ is topologically mixing, there exists $L_{n}$ such that any couple of $W_n$ can be joined by an itinerary of length greater than $L_n$.
	
	Now fix a large number $s_{n+1}>L_n+n$, and construct the word $\w_{n}$ as follows:
	For each pair $(j_1,j_2)\in \{1,\ldots,\eta_n\}^2$, make sure that words $B_{j_1}$
	and $B_{j_2}$ appear in $\w_{n}$, and the distance from the word $B_{j_1}$ to the
	word $B_{j_2}$ is $s_{n+1}$.
	
	Let $\mathfrak{t}$ be the tower in the step $n+1$ in the previous section. Refine $\mathfrak{t}$ according to $\a_n$, and choose one column
	$\mathfrak{c}_{n+1}$ of the resulting tower. Let the base of $\mathfrak{c}_{n+1}$ be $C_{n+1}$. Let $e_{n+1}=\mu(C_{n+1})$. Now we do the following adjustment for the column $\mathfrak{c}_{n+1}$. Copy the name $\w_n$ on some place close to the bottom of the  column $\mathfrak{c}_{n+1}$, for instance we can copy $\w_n$ in the bottom of the second principal subcolumn. We consider towers of level $n+1$ such that the bottom is large enough with respect to the length of $\w_n$ so we can apply the steps I, II and III described in the previous section. This process keeps the good properties related to the uniform rigidity.
	
	As in the previous section, we get a new function $f_{n+1}$ and a corresponding partition $\a_{n+1}$, and we can make sure that
	$$d(\a_{n},\a_{n+1})=\mu(\{f_n\neq f_{n+1}\})<r_{n+1}<\min\Big\{\cfrac{r_n}{2},\cfrac{e_n^2}{4n}\Big\}. $$
	
	By the construction of $\a_{n+1}$, there is a subset $\{U_1^{n+1},\ldots,U^{n+1}_{\eta_n}\}
	\subset$ $\bigvee_{j=0}^{n-1}T^{-j} \a_{n+1}$ such that the $\a_{n}$-name of $U^n_h$ and the $\a_{n+1}$-name of $U^{n+1}_h$ are the same, $\forall 1\le h\le \eta_n$.
	
	
	Let $D_{i_1},D_{i_2},D_{j_1},D_{j_2}\in \{U_1^{n+1},\ldots,U^{n+1}_{\eta_n}\}$, and let their names
	be $B_{i_1}, B_{i_2}$, $B_{j_1}, B_{j_2}\in W_{n}$ respectively, where $1\le i_1,i_2,j_1,j_2\leq \eta_n$. Then by the definition of $\w_{n}$, pairs $(B_{i_1}, B_{j_1})$ and $(B_{i_2}, B_{j_2})$ appear in the word $\w_{n}$. Let $p$ be the position of $B_{i_1}$ in the column $\mathfrak{c}_{n+1}$ and let $r$ be the distance from the position of $B_{i_1}$ to the position of $B_{i_2}$. Then we have:
	$$T^{p-1}C_{n+1}\subset D_{i_1}, T^{p-1+s_{n+1}}C_{n+1}\subset D_{j_1},T^{p-1+r}
	C_{n+1}\subset D_{i_2},T^{p-1+r+s_{n+1}}C_{n+1}\subset D_{j_2}.$$
	It follows that
	\begin{equation*}
	\begin{split}
	& \quad  T^{p-1}C_{n+1}\times T^{p-1+r}C_{n+1}\\ &\subset (D_{i_1}\cap T^{-s_{n+1}}D_{j_1})\times
	(D_{i_2}\cap T^{-s_{n+1}}D_{j_2})\\ & =(D_{i_1}\times  D_{i_2})\cap(T\times T)^{-s_{n+1}}( D_{j_1}\times  D_{j_2})
	\end{split}
	\end{equation*}
	Hence
	\begin{equation*}
	\begin{split}
	& \quad \ \mu\times\mu((D_{i_1}\times  D_{i_2})\cap(T\times T)^{-s_{n+1}}( D_{j_1}\times
	D_{j_2}))\\ & \geq \mu\times\mu(T^{p-1}C_{n+1}\times T^{p-1+r}C_{n+1})\ge e_{n+1}^2>0.
	\end{split}
	\end{equation*}
	Thus $(1)_{n}$ holds. The proof is completed.
\end{proof}

Recall that $\a$ is the partition corresponding to $f$.

\begin{prop}
	The representation $(X_f,\sigma)$ is also weakly mixing.
\end{prop}

\begin{proof}
	We show that for non-empty open sets $A,B,C,D$ there exists $n$ such that $\sigma^n A\cap B$ and $\sigma^nC\cap D$ are non-empty. This is guaranteed by the following property:
	For each $m\geq 0$ and $E_1,F_1,E_2,F_2\in \bigvee_{j=0}^{m-1}T^{-j} \a$, there is some $s$ such that
	$$\mu\times\mu((T\times T)^{s}(E_1\times F_1)\cap(E_2\times F_2)>0.$$
	
	We follow the notations in Lemma \ref{lem4.1}.
	By the definition of $\a$ and Lemma \ref{lem4.1}, there is some large enough $t>m$ such that there are $E'_1,F'_1,E'_2,F'_2\in \bigvee_{j=0}^{m-1}T^{-j} \a_{t}$ such that they have the same
	names with $E_1,F_1,E_2,F_2$ respectively. Choose $C'_1, D'_1, C'_2, D'_2\in \{U_1^{t},\ldots,U^{t}_{\eta_{t-1}}\}
	\subset\bigvee_{j=0}^{t-2}T^{-j} \a_{t}$  such that $C'_1\subset E'_1, D'_1\subset F'_1, C'_2\subset E'_2, D'_2\subset F'_2$. Then there are elements $C_1\subset E_1, D_1\subset F_1, C_2\subset E_2, D_2\subset F_2$ in $\bigvee_{j=0}^{t-2}T^{-j} \a$ such that they have the same names with $C'_1,D_1',C_2',D_2'$ respectively.
	
	By Lemma \ref{lem4.1}-(3),
	$$\mu\times\mu((T\times T)^{s_t}(C'_1\times D'_1)\cap(C'_2\times D'_2))\geq e_{t}^2.$$
	Then by $d(\bigvee_{j=0}^{t-1}T^{-j} \a_t, \bigvee_{j=0}^{t-1}T^{-j} \a)\leq t d(\a_t,\a)<t \sum_{j=t+1}^\infty r_j$, one has that
	\begin{equation*}
	\begin{split}
	& \quad  \mu\times\mu((T\times T)^{s_t}(C_1\times D_1)\cap(C_2\times D_2)) \\ &\geq \mu\times\mu((T\times T)^{s_t}(C'_1\times D'_1)\cap(C'_2\times D'_2))-t\sum_{j=t+1}^{\infty}r_j \\ & \geq e_{t}^2-t\sum_{j=t+1}^{\infty}r_j\geq
	e_t^2-t(r_{t+1}+\frac{r_{t+1}}{2}+\frac{r_{t+1}}{2^2}+\ldots )\\&\geq e_t^2-tr_{t+1}\sum_{j=0}^\infty\frac{1}{2^j}\geq e_t^2-2tr_{t+1}\geq e_t^2/2>0.
	\end{split}
	\end{equation*}
	In particular,
	\begin{equation*}
	\begin{split}
	& \quad  \mu\times\mu((T\times T)^{s_t}(E_1\times F_1)\cap(E_2\times F_2)) \\ &>\mu\times\mu((T\times T)^{s_t}(C_1\times D_1)\cap(C_2\times D_2)) >0.
	\end{split}
	\end{equation*}
	The proof is completed.
\end{proof}

\subsection*{Acknowledgements}
The first author is supported by CONICYT Doctoral fellowship 21110300 and grants Basal-CMM. The second author is supported by NNSF of China (11571335, 11431012, 11171320) and by ``the Fundamental Research Funds for the Central Universities''. The first author thanks the hospitality of University of Science and Technology of China where this research was finished.

\end{document}